\documentclass[a4paper,12pt]{article}
\usepackage[mac]{inputenc}
\usepackage[T1]{fontenc}
\usepackage[frenchb]{babel}
\usepackage{amssymb,amsfonts,amsmath,amsthm}

\newcommand\N{\mathbb N}

\newcommand\R{\mathbb R}
\newcommand\C{\mathbb C}

\newtheorem*{theorem}{Th\'eor\`eme}

\begin{document}

\title{Une famille d'applications lin\'eaires li\'ee \`a l'hypoth\`ese de Riemann g\'{e}n\'{e}ralis\'{e}e}
\author{Eric SAIAS}
\maketitle
\begin{abstract}
Nous montrons qu'une mani\`ere d'aborder la piste presque-p\'eriodique pour l'hypoth\`ese de Riemann g\'en\'eralis\'ee consiste  \`a \'etudier une certaine famille d'applications lin\'eaires. 
\end{abstract}

\section{Introduction}
Pour tous r\'eels $1/2 < \alpha < \beta <1$ et $ \gamma > 0 $, notons  $A_{\alpha, \beta, \gamma} = (a_{\alpha, \beta, \gamma} (m,n))_{m,n\geq 1} $ la matrice infinie d\'efinie par 

$$ a_{\alpha, \beta, \gamma} (m,n)= \frac{( \frac{1}{mn})^\alpha - (\frac{1}{mn})^\beta}{\log(mn)}\ .\  \frac{\sin[\gamma \log(m/n)]}{\log(m/n)}$$
avec les conventions 
$ \frac{( \frac{1}{1})^\alpha - (\frac{1}{1})^\beta}{\log 1} = \beta - \alpha$ et $\frac{\sin[\gamma. 0]}{0}= \gamma.$ 
\newline Les matrices $A_{\alpha, \beta, \gamma}$ sont sym\'etriques d\'efinies positives. On peut donc d\'efinir leur d\'ecomposition de Cholesky  $A_{\alpha, \beta, \gamma} = \,^tU_{\alpha, \beta, \gamma}\ U_{\alpha, \beta, \gamma}$ o\`u $U_{\alpha, \beta, \gamma}$ est 
 triangulaire sup\'erieure infinie avec des r\'eels strictement positifs sur la diagonale. Nous montrons ici qu'une variante de la piste presque-p\'eriodique pour l'hypoth\`ese de Riemann g\'{e}n\'{e}ralis\'{e}e passe par une bonne connaissance des applications lin\'eaires dont les  $U_{\alpha, \beta, \gamma}$ sont les matrices dans la base canonique.  

\section{Applications lin\'eaires $u_K$ et matrices $U_K$}
Pr\'ecisons le lien ici entre application lin\'eaire et matrice. De mani\`ere g\'en\'erale d\'esignons par  $K$ un compact d'int\'{e}rieur non vide du demi-plan $\{s\in\C:  Re(s) >0\}$. On d\'efinit alors le produit scalaire de l'espace de Hilbert $L^2(K)$ par

 $$ \langle f,g \rangle_{L^2(K)} \ = \iint\limits_{\sigma + i \tau \in K}{\overline{f(\sigma+i\tau)}\ g(\sigma+i\tau)}\ d\sigma \ d\tau.$$ 
 
Posons $e_n :=\frac1{n^s}$. La famille $(e_n)_{n\geq1}$ est une famille libre de $L^2(K)$. On peut donc consid\'erer la famille orthonormale $(e'_n(K))_{n\geq1}$, obtenue par orthonormalisation de Gram-Schmidt de la famille $(e_n)_{n\geq1}$.
Soit $E$ l'espace vectoriel des suites $x$ de $\C^{\N^{*}}$ dont l'abscisse de convergence de la s\'erie de Dirichlet $f_x(s) :=\sum_{k=1} ^{\infty} \frac{x(k)}{k^s}$ est $\leq 0$. Pour $x$ dans $E$, la s\'erie $\sum_{k=1} ^{\infty} \frac{x(k)}{k^s}$ converge aussi dans $L^2(K)$. De plus $\Vert f_x(s) \Vert_{L^2(K)}^2 = \sum_{n\geqslant1} \vert \langle e'_n(K), f_x(s)\rangle _{L^2(K)}\vert^2$. On peut donc d\'efinir l'application lin\'eaire 
\begin{align*} u_K: E&\longrightarrow l^2(\N^*) \\
x&\longmapsto (\langle e'_n(K), f_x(s)\rangle )_{n\geq1}
\end{align*}

On appelle base canonique la famille de suites $(\delta_n)_{n\geq1}$ o\`u $\delta_n(k)$ vaut 1 si $k=n$ et 0 sinon. Cette famille constitue une base de $F$, le sous-espace vectoriel de $E$ form\'e des suites \`a support fini. On a  $u_K(F) \subset F$; la restriction de $u_K$ \`a $F$ d\'efinit donc un endomorphisme de $F$: notons  $U_K$  sa matrice dans la base canonique. 
  
  Choisissons maintenant $K=\{s\in\C: \alpha \leq Re(s) \leq \beta \ et \mid Im(s) \mid \leq \gamma\}$. On v\'erifie alors que $\langle e_m, e_n\rangle _{L^2(K)} = 2a_{\alpha, \beta, \gamma} (m,n)$ et que
  $U_K = \sqrt{2}\ U_{\alpha, \beta, \gamma}$.

\section{D'autres notations}
Soit $\chi$ un caract\`ere de Dirichlet. Rappelons que l'on note usuellement $L_{\chi}(s)$ le prolongement m\'eromorphe de $\sum_{k=1} ^{\infty} \frac{\chi (k)}{k^s}$; on note \'egalement $\zeta(s) = L_1(s)$. Par ailleurs,  pour tout r\'eel t, nous notons ici $x_{\chi}(t)$ la suite d\'efinie par $$x_{\chi}(t)(k) = \begin{cases} (-1)^k k^{it}  &si \  \chi  \ est \  principal\\
  \chi(k) \ k^{it} \    
    & si \ \chi \ est \ non \ principal.
  \end{cases}$$
  Posons $S = \{s\in\C: 1/2 < Re(s) < 1\}$ et d\'esignons par $\lambda$ la mesure de Lebesgue sur $\R$.
  Notons enfin 
  \newline $K(\alpha, \beta, \gamma) = \{s\in\C: \alpha \leq Re(s) \leq \beta \ et \mid Im(s) \mid \leq \gamma\}$.
  
  \section{Un crit\`ere de r\'ecurrence dans $l^2$ pour l'hypoth\`ese de Riemann g\'{e}n\'{e}ralis\'{e}e}

\begin{theorem}
Soit $\chi$ un caract\`ere de  Dirichlet. Les assertions suivantes sont \'equivalentes.
\begin{equation} \label{HR1}
 \ L_\chi(s) \  ne\  s'annule\  pas\  dans\  le\  demi-plan \ \{s\in\C:  Re(s) >1/2\}
\end{equation}
 \begin{equation} \label{HR2} \begin{cases}\ &Pour \ tous\  reels \ \alpha, \beta , \gamma \ et \ \varepsilon \ avec \ 1/2 <\alpha < \beta <1, \gamma > 0 \ et  \  \varepsilon >0,  \ on \ a \\
 &\liminf_{T\rightarrow +\infty} \frac{1}{2T}\lambda\{t\in[-T,T]:\Vert u_{K(\alpha, \beta, \gamma)}(x_\chi(t) - x_\chi(0)) \Vert_2 < \varepsilon \} > 0
 \end{cases}
 \end{equation}
\end{theorem}
\begin{proof}
Notons
$$L_{\chi}^{*}(s):= \begin{cases} \sum_{k=1} ^{\infty}  \frac{(-1)^{k}}{k^s} = (2^{1-s} -1)\zeta(s) & si \ \chi \ est \ principal\\
 L_{\chi}(s) & si \ \chi \ est \ non \ principal
\end{cases}$$
On munit l'ensemble $HolS$ des fonctions holomorphes sur $S$ de la topologie de la convergence uniforme sur les compacts de $S$.

 Montrons d'abord que (\ref{HR1}) est \'equivalente \`a 
\begin{equation} \label{HR3} \begin{cases}
&pour \ tout \ voisinage \ V \ de \  L_{\chi}^{*}(s), on \ a \\
&\liminf_{T\rightarrow +\infty} \frac{1}{2T}\lambda\{t\in[-T,T]: L_{\chi}^{*}(s-it)\in V\} > 0
\end{cases}
\end{equation}
Dans le travail de Bagchi \cite{B2}, on peut remplacer les limites sup\'erieures par les limites inf\'erieures sans changer les preuves. En effectuant cette substitution dans le \emph{theorem 3.7}, on obtient directement l'\'equivalence de (\ref{HR1}) et (\ref{HR3}) dans le cas o\`u $\chi$ est non principal. 
  Pour $\chi$ principal, il faut travailler un poil plus. Toutes les assertions (\ref{HR1}) correspondant \`a un caract\`ere principal   sont \'equivalentes entre elles, et en particulier \`a celle correspondant au caract\`ere de Dirichlet modulo 2. En rempla\c{c}ant limite sup\'erieure par  limite inf\'erieure dans le \emph{theorem 4.11} de \cite{B2}, on obtient donc que (\ref{HR1}) entra\^ine (\ref{HR3}). R\'eciproquement supposons (\ref{HR3}). En reprenant alors le raisonnement du \emph{theorem 3.7} de \cite{B1} qui utilise le th\'eor\`eme de Rouch\'e, on obtient l'hypoth\`ese de Riemann pour $\zeta(s)$, et donc aussi (\ref{HR1}).
  
  Soient $K$ et $K'$ deux compacts de $S$ tels que $K' \subset \mathring{K}$. Il existe alors (cf. \emph{lemma 4.8.6} de \cite{B3}) une constante $c>0$ telle que pour toute fonction holomorphe $f(s)$ sur $S$, on a $max_{s\in K'} \vert f(s) \vert \leq c \Vert f \Vert_{L^2(K)}$. Cela permet de v\'erifier que la topologie sur $Hol(S)$ d\'efinie par la famille de normes $\Vert f \Vert_{L^2(K(\alpha, \beta, \gamma))}$ o\`u $1/2 < \alpha < \beta <1$ et $ \gamma > 0 $, co\"incide avec la topologie de la convergence uniforme sur les compacts de $S$. On peut donc r\'ecrire (\ref{HR3}) sous la forme
   \begin{equation*} \begin{cases}\ &Pour \ tous\  reels \ \alpha, \beta , \gamma \ et \ \varepsilon \ avec \ 1/2 <\alpha < \beta <1, \gamma > 0 \ et  \  \varepsilon >0,  \ on \ a \\
 &\liminf_{T\rightarrow +\infty} \frac{1}{2T}\lambda\{t\in[-T,T]:  \Vert L_{\chi}^{*}(s-it) -  L_{\chi}^{*}(s) \Vert_{L^2(K(\alpha, \beta, \gamma))} < \varepsilon \} > 0
 \end{cases}
 \end{equation*}
 Or en notant $K = K(\alpha, \beta, \gamma)$, on a 
  \begin{align*} &\Vert L_{\chi}^{*}(s-it) -  L_{\chi}^{*}(s) \Vert_{L^2(K)} = \Vert \sum_{k=1} ^{\infty}  \frac{(x_\chi(t) - x_\chi(0))(k)}{k^s} \Vert_{L^2(K)} \\ 
 =\  &\Vert \sum_{n\geqslant 1} (u_K(x_\chi(t) - x_\chi(0)))_{n}\  e'_{n}(K) \Vert_{L^2(K)} = \Vert u_K(x_\chi(t) - x_\chi(0)) \Vert_2\ .
\end{align*}
 Cela conclut la preuve du th\'eor\`eme. 
 
\end{proof}

Je tiens ici \`a remercier Pierre Mazet, Damien Simon et Andreas Weingartner, d'une part pour les discussions que nous avons eues sur ce sujet, et d'autre part pour leur \LaTeX aide!


\begin{thebibliography}{99}


\bibitem{B1} B. Bagchi, \textit{A joint universality theorem for Dirichlet L-functions}, Math. Z., \textbf{181} (1982), 319-334.
\bibitem{B2} B. Bagchi, \textit{Recurrence in topological dynamics and the Riemann Hypothesis}, Acta Math. Hung. \textbf{50} (1987), 227-240.
\bibitem{B3} A. Berenstein and R. Gay, Complex Variables An Introduction, Graduate Texts in Mathematics 125, Springer, 1991.

\end{thebibliography}
\end{document}